\documentclass[10pt]{article}
\usepackage{geometry, enumerate, amsthm, amsmath, amssymb, amscd, color}
\usepackage{hyperref}
\usepackage{courier}
\usepackage[all]{xy}

\numberwithin{equation}{section}

\newtheorem{Prop}[equation]{Proposition}
\newtheorem{Thm}[equation]{Theorem}
\newtheorem{Lemma}[equation]{Lemma}

\theoremstyle{definition}
\newtheorem{Def}[equation]{Definition}

\newtheorem{Rem}[equation]{Remark}

\newcommand{\N}{\mathbb{N}}
\newcommand{\Q}{\mathbb{Q}}
\newcommand{\Z}{\mathbb{Z}}
\newcommand{\f}{\mathfrak{F}}

\newcommand{\QQ}{Q}

  \def\o{{\omega}}
    \def\a{{\alpha}}
\def\b{{\beta}}

\begin{document}

\title{The lattice of primary ideals of orders \\ in quadratic number fields}

\author{G. Peruginelli\footnote{Via Pietro Coccoluto Ferrigni 68, 57125 Livorno, Italy. E-mail: g.peruginelli@tiscali.it. } and P. Zanardo\footnote{Department of Mathematics, University of Padova, Via Trieste 63,  35121 Padova, Italy. E-mail: pzanardo@math.unipd.it. Research supported by  ``Progetti di Eccellenza 2011/12" of Fondazione CARIPARO and by the grant "Assegni Senior" of the University of Padova.}}

\maketitle

\begin{abstract}
\noindent
Let $O$ be an order in a quadratic number field $K$ with ring of integers $D$, such that the conductor $\f = f D$ is a prime ideal of $O$,  where $f\in\Z$ is a prime. We give a complete description of the $\f$-primary ideals of $O$. They form a lattice with a particular structure by layers; the first layer, which is the core of the lattice, consists of those $\f$-primary ideals not contained in $\f^2$.  We get three different cases, according to whether the prime number $f$ is split, inert or ramified in $D$.
\end{abstract}

\medskip
\noindent{\small \textbf{Keywords}: Orders, Conductor, Primary ideal, Lattice of ideals.\\ \textbf{MSC Classification codes}:  11R11, 11R04.}


\section{Introduction}

A Dedekind domain is defined as an integral domain in which every ideal can be factored into a product of prime ideals (\cite[\S 6, Ch. IV, p. 270]{ZS}); moreover, this factorization is necessarily unique (\cite[Corollary p. 273]{ZS}). We are interested here in quadratic orders, that is, integral domains $O$ whose integral closure is the ring of integers $D$ of a quadratic number field $K = \Q[\sqrt{d}]$, $d$ a square-free integer. We say that an order is proper if it is not integrally closed, that is, $O\subsetneq D$ (recall that $D$ is a Dedekind domain).  Since a Dedekind domain is necessarily integrally closed, if $O$ is a proper order then there exist ideals of $O$ which cannot be factored into a product of prime ideals. However, since an order is a one-dimensional Noetherian domain, each ideal of $O$ can be written uniquely as
a product of primary ideals (\cite[Theorem 9, Ch. IV, \S 5, p. 213]{ZS}). An order $O$ is determined by its conductor $\f$, defined as the largest ideal of $D$ contained in $O$; equivalently, $\f = \{x\in O : xD\subseteq O\}$. Since $D$ is a finitely generated $O$-module, $\f$ is always non-zero and it is a proper ideal of $O$ if and only if the order is proper. 
Each ideal coprime to the conductor, called {\it regular}, has a unique factorization into prime ideals of $O$ \cite{Cox}. In particular, each regular primary ideal is equal to a power of its radical. Actually, this condition characterizes the regular primary ideals (see \cite[Lemma 2.3]{PSZ}). More interesting is the situation for primary ideals that are non-regular.  In the present paper we focus on the most natural case when the conductor $\f$ is a prime ideal of $O$, so that $\f= fD$, for some prime number $f\in\Z$. In this case, it makes sense to talk about $\f$-primary ideals (i.e., primary ideals whose radical is equal to $\f$). In particular, we will relate the structure of the lattice of $\f$-primary ideals to the splitting type of $f$ in $D$. We reserve further investigations for the general case to a future work. 

Our purpose is to give a detailed description of the structure of the lattice of $\f$-primary ideals of a quadratic order $O$. We get three completely different lattices of $\f$-primary ideals, according to whether $fD$ is a prime ideal in $D$ (inert case), or it is the product of two distinct prime ideals of $D$ (split case), or it is equal to the square of a prime ideal of $D$ (ramified case). However, these lattices have a crucial property in common, namely, a {\it structure by layers}. This means that the structure of the lattice is determined by its first layer, namely the set of $\f$-primary ideals not contained in $\f^2$, which we call {\it basic} $\f$-primary ideals. The remaining part of the lattice is formed by the $n$-th layers of the ideals contained in $\f^n$ and not contained in $\f^{n+1}$, for each $n > 1$, and all these layers reproduce the same pattern of the first layer.

In Sections \ref{general def results} and \ref{intermediatefprimary} we characterize the $\f$-basic ideals. We firstly characterize the $\f$-basic ideals which are also $D$-modules (that is, ideals of $D$). This is a crucial step to get a complete description of the first layer, since every $\f$-basic ideal lies between a suitable $\f$-basic $D$-module $Q$ and $fQ$.  We also identify the $\f$-basic ideals that are principal. We show that there are exactly $f+1$ pairwise distinct intermediate ideals properly lying between $\f$ and $\f^2$. 

In Section \ref{lattice} we examine separately the three cases mentioned above, namely, $f$ inert, split or ramified in $D$, that gives rise to different structures of the corresponding lattices of $\f$-primary ideals.

In the general case of a proper quadratic order $O$ whose conductor $\f$ is not necessarily a prime ideal, we know that $\f$ can be written uniquely as a product of primary ideals $\mathfrak G_1,\ldots,\mathfrak G_s$ whose radicals are distinct maximal ideals $\f_1,\ldots,\f_s$ of $O$. In Remark \ref{finalremark} at the end of the paper we make some initial comments on this case, that we intend to thoroughly investigate in a coming paper.   

\section{General definitions and results}\label{general def results}

In what follows, we will freely use the standard results on rings of integers in quadratic number fields. For example, see \cite{J} and  \cite[Chapter V]{ZS}. As usual, for elements $z \in D$ and ideals $I$, the symbols $\bar{z}$, $\bar{I}$ and $N(z)$, $N(I)$ denote the conjugates and the norms, respectively;  $D^*$, $O^*$ denote the multiplicative groups of the units of $D$ and $O$. If $I$ is an ideal of $O$, $ID$ denotes the extended ideal in $D$, i.e., the ideal of $D$ generated by $I$.  Moreover, in order to simplify the notation, the symbol ``$ \subset$'' will denote proper containment  and as usual ``$I\not\subset J$'' will denote that $I$ is not contained in $J$.  

We fix some notation. Let $d$ be a square-free integer. The ring of integers of $K=\Q(\sqrt{d})$ is equal to 
$D=\Z[\o]$, where either $\o = \sqrt{d}$, when $d \equiv 2, 3$ modulo $4$, or $\o = (1 + \sqrt{d})/2$, when $d \equiv  1$ modulo $4$. In the latter case, we get $\o^2 = \o -(1 + d)/4$. Let now $f$ be a positive integer and $O=\Z[f \o]$ be the unique quadratic order in $K$ such that $[D:O]=f$. For $\a, \b \in O$, we set $(\a, \b) = \a O + \b O$; in general, $(\a, \b)$ strictly contains the $\Z$-module $\a \Z + \b \Z$. By definition, the conductor of $O$ in $D$ is the ideal
$$
\f= \{x \in O : xD \subseteq O\} = f D = f\Z+ f \omega\Z= fO+ f\omega O.
$$
Recall that $\f$ is the largest ideal of $D$ contained in $O$. In particular, $\f$ is not a principal ideal of $O$. A direct check shows that $\f^2= f \f$, hence $\f^k= f^{k-1}\f$ for each $k > 0$. It is also useful to note that
$$
N(\f^k) = |O/\f^k| = |\Z/f^k \Z \oplus  f \o\Z/  f^k \o \Z| = f^{2k-1}.
$$

Since $O/ \f \cong \Z/f\Z$, we immediately see that $\f$ is a prime ideal of $O$ if and only if $f$ is a prime number.  As we have already said in the Introduction, throughout the paper we will assume that $\f$ is a prime ideal; equivalently, $f$ will always denote an assigned prime number. In particular, under the present circumstances, in order to study non-regular ideals, it will make sense to talk about $\f$-primary ideals.  Note that there is no ideal of $O$ lying properly between $fO$ and $\f$, since $[\f : fO] = f$.

It is well known that every primitive ideal of $O$ (i.e. $Q\not\subset nO$, for each $n\geq2$) can be written as 
\begin{equation*}
Q = q \Z + (a+f\omega) \Z=(q,a+f\omega) 
\end{equation*}
where $q,a\in\Z$, such that  $q\Z=Q\cap\Z$ and $q$ divides $N(a+f\omega)$ (see for example \cite{BP1} and \cite{ZZ}). The ideal $Q$ is $O$-invertible if and only if  $(Q : Q) = \{ x \in \Q(\sqrt{d}) : x Q \subseteq Q \} = O$ (\cite[Proposition 7.4]{Cox}). Otherwise, $Q$ is not $O$-invertible and $(Q : Q) = D$ (i.e. $Q$ is a $D$-module). Note that $Q$ is $O$-invertible if and only if $Q \bar{Q} = N(Q) O$; otherwise, $Q \bar{Q} = N(Q) f D$.

\begin{Lemma} \label{generator}
In the above notation, let $\a \in \f \setminus f O$. Then $\f = (f, \a)$.
\end{Lemma}

\begin{proof}
It suffices to show that $f \o \in (f, \a)$. Say $\a = f a + f \o b$, where $a, b \in \Z$ and $f$ does not divide $b$, since $\a \notin f O$. Take $c, k \in \Z$ such that $cb = 1 + f k$. We get 
$$
c \a =  f \o + f (ca + f \o k),
$$ 
whence $f \o \in (f, \a)$, as required.
\end{proof}

Let $Q$ be an ideal of $O$. Using the properties of the norm, it is clear that $Q$ is $\f$-primary if and only if its norm $N(Q)$ is a positive power of $f$. Moreover, if $Q$ is a primitive $\f$-primary ideal of norm $f^k$, then 
\begin{equation}\label{generatorprimary} 
Q = f^k \Z + f \a \Z=f^kO+f\a O,
\end{equation}
for some $\a \in D \setminus O$.

We give a definition which is crucial for our discussion.

\begin{Def} Let $Q\subset O$ be a $\f$-primary ideal and let $t\in O$. We say that $Q$ is $\f$-{\it basic} if $Q \not\subset \f^2= f \f$. We say that $t$ is $\f$-primary if $tO$ is an $\f$-primary ideal. We say that $t$ is $\f$-basic (or simply basic) if $tO$ is a $\f$-primary basic ideal.
\end{Def}

By definition, $\f$ and $fO$ are $\f$-basic ideals; indeed, they are the only $\f$-primary ideals containing $f$, since there are no intermediate ideals between $fO$ and $\f$. An element $t$ in $O$ which is $\f$-primary lies in $\f= f\Z+ f \o \Z$ and therefore has the form $t= fx+ f\o y$, for some $x,y\in\Z$.

The following equivalences for a $\f$-primary ideal $Q$ are straightforward:
\begin{equation}\label{basicprimitive}
Q \textnormal{ is } \f\textnormal{ -basic } \Leftrightarrow Q \textnormal{  is primitive  } \Leftrightarrow Q \not\subset fO.
\end{equation}

Given an $\f$-primary ideal $Q$, the next lemma shows how to associate to $Q$ an $\f$-basic primary ideal in a canonical way.

\begin{Lemma}\label{basicfprimaryideals}
Let $Q$ be a $\f$-primary ideal and let $k=\max\{n\in\N \mid \f^n\supseteq Q\}$. Then we have:

\begin{itemize}
\item[(i)] $Q= f^{k-1} Q'$, where $Q'$ is a $\f$-basic ideal.
\item[(ii)] If $Q/f^m$ is $\f$-basic for some $m > 0$, then $m$ coincides with $k - 1$.
\end{itemize}
\end{Lemma}

\begin{proof}
(i) Since $\f^k = f^{k-1} \f \supseteq Q$, we get $ Q/f^{k-1}= Q' \subseteq \f$. So, as well as $Q$, $Q'$ is $\f$-primary. Moreover, $Q' \not\subset \f^2$, otherwise $f^{k-1} \f^2 = \f^{k+1} \supseteq f^{k-1} Q' = Q$, against the maximality of $k$. We conclude that $Q'$ is $\f$-basic.

\medskip

(ii) From $Q/f^m \subseteq  \f $ we get $Q \subseteq f^{m} \f = \f^{m+1}$, whence $m + 1 \le k$, by the definition of $k$. Moreover, from $Q/f^m \not\subset  \f^2 = f \f $ we get $Q \not\subset f^{m + 1} \f = \f^{m +2 }$, hence $m + 2 > k$.
\end{proof}

Given a $\f$-primary ideal $Q$, the uniquely determined $\f$-basic ideal $Q'$ containing $Q$, as defined in (i) of Lemma \ref{basicfprimaryideals}, is called the {\it basic component} of $Q$. It follows that the lattice $\mathcal L$ of all the $\f$-primary ideals is determined by the lattice $\mathcal L_1$ of the $\f$-basic ideals. In fact, $\mathcal L_1$ will be the first layer of $\mathcal L$, and the other layers of the lattice will be the $\mathcal L_n$ ($n > 0$), consisting of those $\f$-primary ideals contained in $\f^n$ but not in $\f^{n+1}$. By Lemma \ref{basicfprimaryideals}, the elements of $\mathcal L_n$ are obtained by those of $\mathcal L_1$, just multiplying by $f^{n-1}$. Without loss of generality, we focus our attention on $\mathcal L_1$. 
Hence, in what follows, we will investigate the $\f$-basic ideals of $O$.

The next proposition characterizes primary elements in terms of their norms. 

\begin{Prop} \label{norm}\label{primitiveprimary}
Let $t= fx+ f\o y \in \f$ be $\f$-primary, $x,y\in\Z$. Then ${\rm g.c.d.}(x, y) = f^a$, for some $a \ge 0$. Moreover, $t$ is $\f$-basic if and only if $x,y$ are coprime. If the latter conditions hold, then $t$ is an irreducible element of $O$ which is not prime. 
\end{Prop}

\begin{proof} The proof of the first two claims of the statement is straightforward, using the properties of the norm. For the last claim, let us assume, for a contradiction, that $t=rs$, where $r,s\in O$, and neither $r$ nor $s$ is a unit in $O$. Since the norm is a multiplicative function on $O$, $r,s$ are $\f$-primary elements. In particular, $r,s\in\f$. But then $t=rs\in\f^2$, contradiction. Moreover, $tO$ is not a prime ideal, since it is strictly contained in the conductor $\f$ (the only prime ideal containing $t$), which is not principal.
\end{proof}
 
Let $t\in O$ be an $\f$-primary element, $t=fx+f\omega y$, $x,y\in\Z$. Note that $t$ is in $fO$ if and only if $f$ divides $y$, since $t=f(x+\o y)$ and $x+\o y \in O$ if and only if $f \mid y$. So, by  (\ref{basicprimitive}), if $y \notin f \Z$ then $t$ is $\f$-basic. Note also that in this case $x,y$ are coprime, since $f$ is the only common prime factor of $x$ and $y$. If $t$ is a basic element and $t\in fO$, then $tO=fO$, that is, $t$ and $f$ are associated in $O$.

However, for a basic element $t$, it is possible that $t\notin fO$, but $\f^2\subset tO\subset \f$. We will see in the next section that this happens precisely when $t$ and $f$ are associated in $D$ but not in $O$ (Lemma \ref{units2}).

\section{Intermediate $\f$-primary ideals}\label{intermediatefprimary}

Throughout this section, given a basic $\f$-primary ideal $Q\subset O$ different from $fO$, by (\ref{generatorprimary}) and (\ref{basicprimitive})  we may suppose that $Q=(f^k,f\alpha)$, where $f^k=N(Q)$ and $\alpha\in D\setminus O$.

The following easy lemma determines whether an ideal of $O$ is a $D$-module or not.  If $I$ is an ideal of $O$ and $ID$ is the extended ideal in $D$, $[ID:I]$ denotes the index of $I$ in $ID$ as abelian groups. 


 \begin{Lemma} \label{units1}\label{index}
Let $I$ be an ideal of $O$. 
\begin{itemize}
\item[i)] If $z I \subseteq I$ for some $z \in D \setminus O$, then $I = ID$. 
\item[ii)] If $I\subset ID$, then $[ID:I]=f$.
\end{itemize} 
\end{Lemma} 

\begin{proof}
i). By the preliminaries of Section \ref{general def results}, $(I:I)$ is equal to $O$ if and only if $I$ is not a $D$-module. Hence, $(I:I)=D$, which proves the claim.
 
ii). Let $\alpha=\sum_i a_i \beta_i\in ID$, for  some $a_i\in I$ and $\beta_i\in D$. Then $f\alpha=\sum_i a_i f\beta_i$ is an element of $I$, since each $f \beta_i$ is in $O$. In particular, $fID\subset I\subset ID$, where the inclusions are strict, since $I$ is not a $D$-module. Since $f$ is a prime number and the index of $fID$ in $ID$ is $f^2$, it follows that the index of $I$ in $ID$ is $f$.
\end{proof}
 
The next proposition characterizes the $\f$-basic ideals of $O$ that are also $D$-modules. This kind of ideals will be crucial in the description of the lattice of $\f$-basic ideals. This result also follows from \cite[p. 34]{BP}. We give a direct proof for the sake of completeness.

\begin{Prop}\label{D-ideal}
Let $Q = (f^k, f \a)$ be a $\f$-basic ideal different from $fO$. Then $Q$ is a $D$-module if and only if $f^{k - 1}$ divides $N(\a)$.
\end{Prop}

\begin{proof}
Recall that $Q$ is a $D$-module if and only if $Q$ is not $O$-invertible (see Section \ref{general def results}). We have $Q \bar{Q} = (f^{2k}, f^{k+1} \a, f^{k+1} \bar{\a}, f^2 N(\a))$. If $Q$ is a $D$-module, then $Q \bar{Q} = f^{k+1} D$ and therefore $f^{k-1} \mid N(\a)$. If $f^{k-1} \mid N(\a)$ then $Q \bar{Q} \subset f^{k+1} O$, hence $Q \bar{Q} \ne f^k O$, so $Q$ is not $O$-invertible.
\end{proof}

We describe now the primary ideals lying in between a given $\f$-primary ideal $Q$ and $fQ$, according to whether $Q$ is a $D$-module or not. 

\begin{Thm} \label{intermediate1}\label{intermediate2}
Let $Q = (f^k, f \a)$ be a $\f$-basic ideal different from $fO$. 

\begin{itemize} 

\item[(i)] $\f^k$ is the minimum power of $\f$ contained in $Q$.

\item[(ii)] If $Q$ is a $D$-module, then there are exactly $f+1$ ideals of $O$ lying properly between $Q$ and $fQ$, namely the pairwise distinct ideals
$$
J = (f^k, f^2 \a) ; \ J_a = (f^{k+1}, af^{k} + f\a),  \quad a = 0, 1, \dots, f-1.
$$

\item[(iii)] If $Q \ne QD$, then there is a unique ideal of $O$ lying properly between $Q$ and $fQ$, namely $J = (f^k, f^2 \a) = f QD$.

\end{itemize}
\end{Thm}

\begin{proof}
(i) Recall that $\a \notin O$, since $Q \not\subset f O$, so that $\f = (f, f\a)$ (Lemma \ref{generator}). Then $Q \supseteq (f^k, f^k \a) = f^{k-1} \f = \f^k$, where the equality holds if and only if $k = 1$. Since $f^{k-1} \in \f^{k - 1} \setminus Q$, $k$ is the minimal integer such that $\f^k\subseteq Q$.

(ii) Let $\a = a_1 + \o a_2$, where $a_2 \notin f \Z$, since $\a \notin O$. Since $Q/fQ \cong \Z/f \Z \oplus \Z/f \Z$ (as abelian groups) and $\Z/f \Z \oplus \Z/f \Z$ has exactly $f + 1$ proper non-zero subgroups, it suffices to show that the ideals $J$, $J_a$ ($a = 0, \dots, f-1$) are pairwise distinct and lie properly between $Q$ and $f Q$.

It is clear that the ideals $J$, $J_a$, $0 \le a \le f-1$ lie between $Q$ and $fQ = (f^{k+1}, f^2 \a)$. We firstly verify that these ideals are pairwise distinct. 

Let us suppose that $J_a = J_b$. Then we get the equality
$$
f(af^{k-1} + \a) = (x_0 + x_1 f \o) f^{k+1} + (y_0 + y_1 f \o)(f(bf^{k-1} + \a)) .
$$ 
for suitable $x_0, x_1, y_0, y_1 \in \Z$.
It follows that
$$
af^{k-1} + \a - x_0  f^{k} - y_0(bf^{k-1} + \a) \in  \o Q \subseteq Q,
$$
where $\o Q\subseteq Q$ since $Q$ is a $D$-module. The above relation yields $(1 - y_0)\a \in O$, so $1 - y_0 \in f \Z$, since $a_1 \notin f \Z$. Then we get $af^{k-1}  - y_0bf^{k-1} \in Q$, hence $a - y_0 b \in f \Z$, by the minimality of $k$. We conclude that
$$
1 \equiv y_0, \quad a \equiv y_0 b \quad {\rm mod} \ f,
$$
so $a \equiv b$ modulo $f$, and therefore $a = b$, since these integers both lie in $\{0, 1, \dots , f-1 \}$. We remark that we have actually proved that $J_a \not\subset J_b$ whenever $a \ne b$.

Since $J_a \not\subset fO$, for every $a \le f-1$, we get $J_a \ne J \subset f O$, and $J_a \supset f Q$. Moreover $Q \supset J$, since $Q \not\subset fO$, and $J \supset fQ$, since $f^{k-1} \notin Q$ yields $f^k \in J \setminus fQ$. 

It remains to show that $J_a \ne Q$, for $a = 0, \dots, f-1$. Assume, for a contradiction, that $J_b = Q$ for some $b \le f-1$. Then we get $J_a \subseteq Q = J_b$ for every $a \ne b$, which is impossible, as remarked above. 

(iii) Under the present circumstances, we get $Q \supset f QD \supset f Q$, since $Q$ is not a $D$-module. Let $J$ be an $\f$-primary ideal properly lying between $Q$ and $fQ$. Since $Q$ is not a $D$-module, $Q$ is an invertible $O$-ideal (see Section \ref{general def results}). Therefore, $I = J Q^{-1}$ is an $\f$-primary ideal of $O$, so we get $J = QI \subseteq Q \f = f QD$. Hence, we actually get the equality $J = fQD$, since $[Q: f Q] = f^2$. In particular, $J = (f^k, f^2 \a)$.
\end{proof}

In particular, the preceding theorem allows us to determine the ideals lying between $\f$ and $\f^2$, since $\f$ is a $D$-module and $\f^2 = f \f$.

In the next lemma, we determine the intermediate ideals that are principal, or, equivalently, the basic elements $t\in O$ such that $\f^2\subset tO\subset \f$.  
 
 \begin{Lemma} \label{units2}
 A principal ideal $tO$ lies properly between $\f$ and $\f^2$ if and only if $t = f w$, for a suitable unit $w$ of $D$. Moreover $f w O = f w' O$ if and only if $w/w' \in O$.
 \end{Lemma}
 
 \begin{proof}
Assume that $\f \supset t O \supset \f^2$. The extended ideals satisfy $\f \supseteq t D \supset \f^2$, where the second containment is strict, since $t D \supset tO \supset \f^2$. Since $|\f/\f^2| = f^2$, we get $t D = \f = f D$, which is possible only if $t = f w$ for some unit $w$ of $D$. Conversely, for every unit $w$ of $D$, from $\f \supset fO \supset \f^2$ we get $w \f = \f \supset f w O \supset w \f^2 = \f^2$. The last statement is immediate. 
 \end{proof}

In particular, Lemma \ref{units2} implies that the number of principal $\f$-primary ideals between $\f$ and $\f^2$ is equal to $|D^*/O^*|$. This last quantity depends on how the prime $f$ splits in $D$.
 
\begin{Prop} \label{three cases}

Let $\tau = |D^*/O^*|$. Then we have
\begin{itemize}
\item[i)] if $f$ is inert in $D$, then $\tau \mid f+1$.

\item[ii)] if $f$ is split in $D$, then $\tau \mid f-1$.

\item[iii)] if $f$ is ramified in $D$, then $\tau \mid f$.
\end{itemize}
\end{Prop}

\begin{proof}
Since $f$ is prime, $O/\f\cong \mathbb{F}_f$, the finite field with $f$ elements. In particular, the group of units of $O/\f$ has cardinality $f-1$.  The residue ring $D/\f$ is isomorphic either to $\mathbb{F}_{f^2}$ (inert case), $\mathbb{F}_{f}\times \mathbb{F}_{f}$ (split case) or to a finite local ring with  principal maximal ideal (ramified case).  In each of the three cases, the group of units of $D/\f$ has cardinality equal to $f^2-1$, $(f-1)^2$ and $f^2-f$, respectively.

The canonical ring homomorphism $\pi:D\twoheadrightarrow D/\f$ induces a group homomorphism $\pi^*: D^*\to (D/\f)^*$ (which is not necessarily surjective). We have an induced  group homomorphism: $D^*/O^*\to(D/\f)^*/(O/\f)^*$, $u+O^*\mapsto \pi^*(u)+(O/\f)^*$. 
We claim that the latter group homomorphism is injective. In fact, if $\pi^*(u)\in (O/\f)^*$, then $\pi(u)\in O/\f$, so we get $u \in O^*$, since $\pi^{-1}(O/\f)=O$. It follows that $\tau = |D^*/O^*|$ divides the cardinality of ${(D/\f)^*}/{(O/\f)^*}$, which in the three cases is equal to: i) $f+1$ (inert), ii) $f-1$ (split), iii) $f$ (ramified).
\end{proof}

\begin{Rem}\label{numberintermediateDmodules}
We note that the same conclusion of Proposition \ref{three cases} can be obtained by means of a well-known formula that gives the class number of $O$ in terms of the class number of $D$ (see \cite[p. 146-148]{Cox}). By Theorem \ref{intermediate2}, there are $f+1$ ideals properly lying between $\f$ and $\f^2$. In each of the three cases mentioned above, the number of these intermediate ideals of $O$ that are $D$-modules is:
\begin{itemize}

\item[i)] inert case: there is no intermediate $D$-module, since  there are no $D$-modules between $\f=P$ and $\f^2=P^2$.

\item[ii)] split case: $2$; the only $D$-modules between $\f=P\overline{P}$ and $\f^2=P^2\overline{P}^2$ are $P^2\overline{P}$ and $P\overline{P}^2$.

\item[iii)] ramified case: $1$; the only $D$-module between $\f=P^2$ and $\f^2=P^4$ is $P^3$.

\end{itemize}
Hence, $\tau = |D^*/O^*|$ divides the number of ideals properly between $\f$ and $\f^2$ that are not $D$-modules ($f+1$, $f-1$ and $f$, resp.), and this last number is equal to the cardinality of ${(D/\f)^*}/{(O/\f)^*}$. 
\end{Rem}

This last fact is an evidence of the following general result. We recall that an action of a group $G$ on a set $S$ is free if the stabilizer of each element $s\in S$ is trivial, that is, ${\rm Stab}(s)=\{g\in G \mid gs=s\}=\{1\}$.

\begin{Prop}\label{freeaction}
The multiplicative group $(D/\f)^*/(O/\f)^*$ acts freely on the set of the ideals $I$ of $O$ that lie properly between $\f$ and $\f^2$ and are not $D$-modules. 
\end{Prop}

\begin{proof}
Let $\mathcal{I}$ be the set of ideals of $O$ lying properly between $\f$ and $\f^2$. The set $\mathcal{I}$ is in one-to-one correspondence with the set $[\mathcal{I}]$ of proper non-zero ideals of $O/\f^2$, by the canonical map $I\mapsto I+\f^2=[I]$. Recall that $\f/\f^2$ is in a natural way a $(D/\f)$-module, and so also a $(O/\f)$-module.

For any assigned $[z]\in (D/\f)^*$ and $[I]\in [\mathcal{I}]$, we set $[z]\cdot[I]=[zI]$. Since $[I]$ is a $O/\f$-module contained in $\f/\f^2$, it is straightforward to see that $[zI]$ is also a $O/\f$-module, contained in $[z]\cdot \f/\f^2=\f/\f^2$, where the last equality holds since $[z]$ is a unit in $D/\f$. We have thus defined an action of $(D/\f)^*$ on $[\mathcal{I}]$. In particular, every element $[I]$ of $[\mathcal{I}]$ is fixed by the elements of the subgroup $(O/\f)^*\subset (D/\f)^*$, i.e., $[z] \cdot [I] = [I]$, for every $[z]\in (O/\f)^*$. Hence we have an induced natural action of the group $G=(D/\f)^*/(O/\f)^*$ on $[\mathcal{I}]$. We can partition $\mathcal{I}$ into the union of the subset $\mathcal{I}_D$ of the ideals that are also $D$-modules and the complementary subset $\mathcal{I}_O$. The set $[\mathcal{I}]$ is therefore partitioned by the natural map into the union of the set $[\mathcal{I}]_{D/\f}$ of $O/\f$-modules which are also $D/\f$-modules and the subset $[\mathcal{I}]_{O/\f}$ of $O/\f$-modules which are not $D/\f$-modules. 
By Lemma \ref{units1}, for any assigned $I\in \mathcal{I}_O$ and $z\in D\setminus O$, we get $zI\not\subset I$. Hence, the sets $[\mathcal{I}]_{D/\f}$ and $[\mathcal{I}]_{O/\f}$ are characterized as follows:
\begin{align*}
[\mathcal{I}]_{D/\f}=&\{\bar I\in [\mathcal{I}] \mid \forall g\in G, g\cdot [I]=[I]\}\\
[\mathcal{I}]_{O/\f}=&\{\bar I\in [\mathcal{I}] \mid \forall g\in G, g\not=1, g\cdot [I]\not=[I]\}.
\end{align*}
Then $[\mathcal{I}]_{D/\f}$ is precisely the subset of $[\mathcal{I}]$ of the fixed elements under the action of $G$ and $[\mathcal{I}]_{O/\f}$ is the subset of elements whose stabilizer under the action of $G$ is trivial. We conclude that $G$ acts freely on the subset $[\mathcal{I}]_{O/\f}$. 
\end{proof}

By the above proposition, the cardinality of $G$ divides the cardinality of $[\mathcal{I}]_{O/\f}$. However, in the present case where the conductor is $fD$, $f\in\Z$ a prime number, we know by the above discussion that the two cardinalities coincide in all the three possible cases, inert, split and ramified.

\section{The lattice of basic ideals}\label{lattice}
 
In the present section we analyze separately the lattice of $\f$-basic ideals, in each of the three cases that may appear, namely: $f$ inert, split or ramified in $D$, respectively.

\subsection{Inert case} 
 
The next theorem gives a complete description of the lattice of the $\f$-basic ideals of $O = \Z [f \o]$, in the case when $f$ is a prime element of $D = \Z [\o]$. 

\begin{Thm} \label{f-prime1}
Suppose $\f=fD$ is a prime ideal of $D$. 
Then every basic $\f$-primary ideal of $O$ contains $\f^2$, and lies in the following set of pairwise distinct ideals
$$
\mathcal J = \{ (f, f^{2}  \o) , (f^{2}, f(a + \o)) : 0 \le a < f \}.
$$

\end{Thm}

\begin{proof}
Let $Q$ be a basic ideal. The extended ideal $QD$ is equal to $\f$, since $Q$ is $\f$-primary ($\f$ is the only prime ideal of $O$ that contains $Q$, hence the only prime ideal of $D$ that contains $Q$) and $Q$ is not contained in $\f^2$ by definition. It follows by Lemma \ref{index} that $fQD=f\f=\f^2\subset Q$. By Theorem \ref{intermediate2}, $Q$ lies in the set $\mathcal J$.
\end{proof}

The number of principal basic $\f$-primary ideals is exactly equal to the number of distinct non-associated basic elements of $O$,  which is equal to $|D^*/O^*|$, by Lemma \ref{units2}. Moreover, their norm is equal to $f^2$, since the ideal they generate lies in between $\f$ and $\f^2$. 

The following diagram represents the lattice of $\f$-primary ideals in the inert case. We recall that  only the powers of $\f$ are $D$-modules (see Remark \ref{numberintermediateDmodules}). For this reason, all the proper intermediate ideals are $O$-invertible (see Section \ref{general def results}).
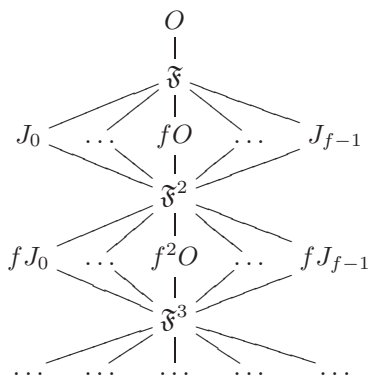
$$\xymatrix@-1.4pc{
&&O&&\\
&&\f\ar@{-}[u]&&\\
J_0\ar@{-}[urr]&\ldots\ar@{-}[ur]&fO\ar@{-}[u]&\ldots\ar@{-}[ul]&J_{f-1}\ar@{-}[ull]\\
&&\f^2\ar@{-}[u]\ar@{-}[ul]\ar@{-}[ur]\ar@{-}[ull]\ar@{-}[urr]&&\\
fJ_0\ar@{-}[urr]&\ldots\ar@{-}[ur]&f^2O\ar@{-}[u]&\ldots\ar@{-}[ul]&f J_{f-1}\ar@{-}[ull]\\
&&\f^3\ar@{-}[u]\ar@{-}[ul]\ar@{-}[ur]\ar@{-}[ull]\ar@{-}[urr]\ar@{-}[d]\ar@{-}[dl]\ar@{-}[dr]\ar@{-}[dll]\ar@{-}[drr]&&\\
\ldots&\ldots&\ldots&\ldots&\ldots
}
$$

\subsection{Split case}

Throughout this section, we assume that $\f = f D$ splits as an ideal of $D$, say $f D = P \bar P$, where $P \ne \bar{P}$ are prime ideals of $D$ of norm $f$, both of which lie above $\f$, considered as an ideal of $O$. Note that $P$ is principal if and only if $f$ is not irreducible in $D$ (recall that $f$ is always irreducible in $O$, by Proposition \ref{primitiveprimary}). However, some power of $P$ is a principal ideal of $D$, since the class group of $D$ is finite. For the remainder of this section, we will denote by $m$ the order of $P$ in the class group of $D$ (i.e., the minimum power $m$ of $P$ such that $P^m$ is principal), and by $\b\in D$ a fixed generator of $P^m$.

\begin{Lemma} \label{beta}
In the above notation, $\beta^n \notin O$ for every $n > 0$.
\end{Lemma}

\begin{proof}
Assume, for a contradiction, that $\beta^n \in O$. Then $\beta^n \in O\cap P=\f$. It follows that $\beta^n D = P^{mn} \subseteq \f = P \bar P \subset \bar P$, whence $P \subseteq \bar P$, impossible.
\end{proof}

The following theorem describes all the $\f$-basic elements of $O$: it turns out that they are associated to the elements $t_n=f\beta^n$, for some $n\in\N$. In particular, in the split case, unlike the inert case, there are basic elements of arbitrary large norm, so, they are infinitely many.
 
\begin{Thm}\label{classificationbasicfprimaryelements}
For each $n\in\N$, let $t_n = f \b^n$. An element $t\in O$ is basic if and only of $t$ is associated in $D$ either to $t_n$ or its conjugate, for some $n\in\N$. Moreover, the principal ideals $t_n w O, \bar{t_n}w' O$, for $n>0$ and $w,w'\in D^*$, $w/w'\notin O$, are pairwise incomparable and do not contain $\f^2$.  
\end{Thm}
 
\begin{proof}
Since $N(t_n) = f^2 N(\b^n) = f^{mn + 2}$, every  element $t_n$ is $\f$-primary. Moreover, note that $ t_n \notin \f^2 = f \f$, since $t_n/f = \b^n \notin \f$, so that $t_n$ is $\f$-basic, for every $n \ge 0$. Pick now two distinct non-negative integers $n$, $m$, with $n = m + h$, $h > 0$. Since $t_n/ t_m = \b^h \notin O$ and $t_m/ t_n = \b^{-h} \notin O$, it follows that the ideals $t_n O$, for $n \ge 0$, are pairwise incomparable. Finally, since $t_n$ has norm strictly greater than $f^2$, for $n>0$, $\f^2$ is not contained in $t_n O$.

Conversely, let $t$ be a basic element of $O$ of norm $f^{s +2}$, $s \geq 0$. Since $t$ is $\f$-basic, $P, \bar P$ are the only prime ideals of $D$ above $t D$. Then we get
$$
tD=P^k  {\bar P}^h, \quad h, k > 0.
$$
 Moreover, since $t \notin \f^2 = P^2 {\bar P}^2$, the integers $h, k$ are not both $> 1$. Let us assume that $h = 1$, whence $t D = f P^{k-1}$. Then $P^{k-1}$ is principal, hence $k - 1 = mn$, for some positive integer $n$. It follows that $N(t) = f^{s + 2} = f^2 N(P^{mn}) = f^{mn + 2}$, so, $s = m n$.  Now, we have $ t D = fP^{mn} = f \b^n D = t_n D $, which is possible only if $t = t_n w$, for some $w \in D^*$. In the case $k = 1$ we symmetrically get $t = {\bar t}_n w$ for some $w \in D^*$.

Finally, if $t_h w O = {\bar t_k} w' O$, then $h=k$ otherwise $t_h, t_k$ have different norms and we get that some power of $\b$ is in $O$, which is impossible by Lemma \ref{beta}. Moreover, $t_h w O = t_k w' O$ implies $h = k$ as before, hence we also get $w /w' \in O$.
\end{proof}

Our next step is to classify the non-principal basic $\f$-primary ideals.

We recall that a Special PIR (Special Principal Ideal Ring) $R$ is a principal ideal ring with a unique prime ideal $M$, such that $M$ is nilpotent (see \cite[p. 245]{ZS}). So, in the case when $M = pR$, for some $p\in R$, we get $p^n=0$ for some $n>1$. Note that a Special PIR is a chained ring, i.e., the ideals are linearly ordered.

The next lemma gives all the basic $\f$-primary ideals that contain some $\f$-basic element.

\begin{Lemma}\label{SPIR}\label{basicfprimary1}
 The quotient ring $O/t_nO$ is a Special PIR for every $n \ge 0$. In particular, the ideals (necessarily $\f$-primary) that contain $t_n O$ are equal to $(f^i , t_n)$, for $i=1 ,\ldots, mn + 2$,  and their norm of $(f^i,t_n)$ is $f^i$. 
\end{Lemma}

\begin{proof}
The claim is immediate when $t_n = t_0 = f$, since $\f/fO$ is the unique nonzero proper ideal of $O/fO$, it is generated by $f \o + f O$, and $(\f/fO)^2 = 0$, since $\f^2 \subset f O$. Note that, if $I$ is an ideal of $O$ containing $t_n$, then $I$ is basic $\f$-primary, since any prime ideal containing $I$ must contain the $\f$-basic element $t_n$. In particular, $O/t_nO$ has a unique maximal ideal, equal to $\f/t_nO$. Since $\f=( f , t_n)$ by Lemma \ref{generator}, it follows that $\f/t_nO$ is a principal ideal of $O/t_nO$, generated by $f +t_nO$. From this fact, it is not difficult to see that every nonzero ideal of $O/t_nO$ is principal, generated by some $f^i + t_n O$, for some $1 \le i \le mn + 1$ (see \cite[Proposition 4]{Hf}, for example). Indeed, $f^h \in t_n O$ if and only if $h \ge mn + 2$, since $N(t_n)= f^{mn+2}$.

Since $f^i$ is the least power of $f$ contained in the basic ideal $(f^i, t_n)$ (which therefore is primitive by (\ref{basicprimitive})), the last claim follows by the preliminaries in Section \ref{general def results}. 
\end{proof}

\begin{Prop} \label{D-modules}
Let $t\in O$ be a basic $\f$-primary element of norm $f^m$, and let $i\in\N$ be such that $i<m$. Then the ideal $I = (f^i,t)$ of $O$ is a $D$-module, equal either to $P^i\bar{P}$ or $P \bar{P}^i$. In particular, we get $(f^i, t_i) = (f^i, t_n)$, for every $n \ge i$. 
\end{Prop}

\begin{proof}
Since $f^{i + 1} \mid N(t)$, we get $I = ID$, by Proposition \ref{D-ideal}. Without loss of generality, we suppose that $tD = P^{m-1}\bar{P}$ (see the proof of Theorem \ref{classificationbasicfprimaryelements}). Since $D$ is a Dedekind domain, $f^i D+tD$ is the greatest common divisor of $f^i D$ and $tD$, so it is equal to $P^i\bar{P}$, since $f^i D=(P\bar{P})^i$. Hence, $I=ID=P^i\bar{P}$. The last claim follows immediately, since $f^i$ divides $N(t_n) = f^{nm + 2}$ for every $n \ge i$.
\end{proof}

For every $k \ge 1$, let $\QQ_k= (f^k, t_k) = P^{k}\bar{P}$; in this notation, $\QQ_1 = \f$.

The next theorem gives a description of the ideals of $O$ that contain a basic element.

\begin{Thm} \label{Qk}
\begin{itemize}

\item[(i)] Let $Q$ be a $\f$-basic ideal. Then there exists $k \geq 1$ such that $f \QQ_k\subset Q\subseteq \QQ_k$.

\item[(ii)] The ideals $\QQ_k = (f^{k}, t_k)$, for $k\in\N$, are pairwise distinct.

\item[(iii)] An ideal $Q$ of $O$ contains $\QQ_k$ is and only if $Q \in \{\QQ_i : i=0,\ldots,k\}$.

\item[(iv)] If $Q$ contains a basic element and it is not principal, then either $Q= \QQ_k$ or $Q={\bar \QQ}_k$ for some $k \in\N$.

\end{itemize}
\end{Thm}

\begin{proof}
(i) Since $Q$ is basic, as in the proof of Theorem \ref{classificationbasicfprimaryelements}, we have $QD=P^{k}\bar{P}=\QQ_k$, for some $k\geq 1$ (or its conjugate), so $Q\subseteq \QQ_k$. By Lemma \ref{index}, either $Q=\QQ_k$ or $[\QQ_k:Q]=f$. In each case, we get $f\QQ_k\subset Q\subseteq \QQ_k$.

(ii) By Proposition \ref{D-modules}, we get $\QQ_k = P^k \bar P$ (and not the conjugate, since $\b^k \in P \setminus \bar P$). Hence the $\QQ_k $'s are pairwise distinct, as $k$ ranges in $\N$.

(iii) For $0 \le i \le k$, by Proposition \ref{D-modules} we get $
\QQ_i = (f^{i},t_i) = (f^{i},t_k)\supseteq (f^{k}, t_k) = \QQ_k$. 
 Conversely, if $I \supseteq \QQ_k$, then $I$ contains $t_k$, hence, by Lemma \ref{basicfprimary1}, we get $I=(f^j,t_k)$, for some $j \in \{1,\ldots, k+1\}$, so $I = (f^j, t_k) = (f^j, t_j) = \QQ_j$. 
 
(iv) This follows from (ii) and its proof, possibly replacing $\QQ_i$ with their conjugates.
\end{proof}

In order to complete the description of the lattice of $\f$-basic ideals, it remains to find the basic ideals of $O$ that do not contain a $\f$-basic element.

\begin{Thm} \label{not containing F-basic}
 Let $Q$ be a basic $\f$-primary ideal not containing any basic element. Then 
 \begin{itemize}
\item[(i)] $Q$ lies properly between $\QQ_k$ and $f \QQ_k$, for some $k > 0$; 
 \item[(ii)] $Q = (f^{k+1}, a f^{k} + t_k)$ for some $1 \le a \le f - 1$; 
 \item[(iii)]  $Q$ does not contain any other basic $\f$-primary ideal;
 \item[(iv)] $Q$ is an invertible ideal of $O$. 
\end{itemize} 
 \end{Thm}

\begin{proof}
(i) The ideal $Q$, being $\f$-basic, must lie between some $\QQ_k$ and $f \QQ_k$ by (i), and it is different from $Q_k$, since it does not contain basic elements.

(ii) This follows from Theorem \ref{intermediate1}, since necessarily $Q$ is different from $ (f^{k}, ft_k) = f\QQ_{k-1}$, which is not a $\f$-basic primary ideal, and from $(f^{k+1}, t_k)$, which contains the $\f$-basic element $t_k$.

(iii) Let $Q'$ be a basic ideal contained in $Q = (f^{k+1}, a f^{k} + t_k)$. Then $Q'$ cannot contain a $\f$-basic element, hence, by (ii) we get $Q' = (f^{h+1}, b f^{h} + t_h)$, for some $h > 0$, $b \in \{1, \dots, f-1 \}$. Let us assume, for a contradiction, that $Q \ne Q'$, so $Q \supset Q'$. It follows that $h > k$. Then we readily see that $Q' \subset Q$ if and only if $t_h \in Q$, impossible, since $t_h$ is $\f$-basic.

(iv) Let $f \gamma = a f^{k} + t_k = f(a f^{k-1} + \b^k)$. By Proposition \ref{D-ideal}, it suffices to show that $f^{k}$ does not divide $N(\gamma)$. We get $
N(\gamma) = a^2 f^{2k-2} + af^{k-1}(\b^k + \bar{\b}^k) + f^{mk}$. Since $f$ does not divide the trace of $\b^k$ (otherwise $\b^k \in f D = \f$, impossible), we see that $N(\gamma) = f^{k-1} b$, where $b \notin f \Z$.
\end{proof}

Note that an ideal $Q$ satisfying the hypothesis of the previous theorem, is not a $D$-module. The converse of Theorem \ref{not containing F-basic}, iv) is false: consider any principal $\f$-primary ideal generated by a basic element. Therefore, the basic ideals that are invertible are either principal, necessarily generated by a $\f$-basic element, or they do not contain any $\f$-basic element.
\begin{Rem}
Let $k\in\N$. By Theorem \ref{classificationbasicfprimaryelements}, there exist principal intermediate ideals between $Q_k$  and $fQ_k$ if and only if $Q_k$ is principal as an ideal of $D$, generated by a $\f$-basic element of $O$. In fact, if $fQ_k\subset tO\subset Q_k$ then we have $tD=Q_k$. Conversely, if $Q_k\subseteq \f=f D$ is principal, then $Q_k$ is generated by an element of the form $f\beta$, for some $\beta\in D\setminus O$. Hence, $f\beta O$ is an intermediate ideal between $fQ_k$ and $Q_k$. Moreover, as we saw in the proof of Theorem \ref{classificationbasicfprimaryelements}, the last condition holds if and only if $m$ divides $k-1$. For such $k$'s, there are $\tau=[D^*:O^*]$ intermediate principal ideals between $Q_k$  and $fQ_k$ (essentially by the same phenomenon of Lemma \ref{units2}).
\end{Rem}

The diagram below represents the lattice of $\f$-primary ideals in the split case. 
$$
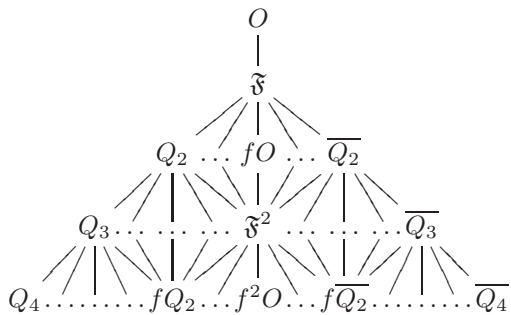
\xymatrix@R-1.1pc@C-2.4pc{
&&&&&&&&O&&&&&&\\
&&&&&&&&\f\ar@{-}[u]\ar@{-}[dl]\ar@{-}[dr]&&&&&&\\
&&&&&&\QQ_2\ar@{-}[urr]&\ldots&fO\ar@{-}[u]&\ldots&\overline{\QQ_2}\ar@{-}[ull]&&&&\\
&&&&\QQ_3\ar@{-}[urr]&\ldots\ar@{-}[ur]&\ldots\ar@{-}[u]&\ldots\ar@{-}[ul]&\f^2\ar@{-}[u]\ar@{-}[ul]\ar@{-}[ull]\ar@{-}[ur]\ar@{-}[urr]&\ldots\ar@{-}[ur]&\ldots\ar@{-}[u]&\ldots\ar@{-}[ul]&\overline{\QQ_3}\ar@{-}[ull]&&\\
&&\QQ_4\ar@{-}[urr]&\ldots\ar@{-}[ur]&\ldots\ar@{-}[u]&\ldots\ar@{-}[ul]&f\QQ_2\ar@{-}[u]\ar@{-}[ul]\ar@{-}[ull]\ar@{-}[ur]\ar@{-}[urr]&\ldots\ar@{-}[ur]&f^2O\ar@{-}[u]&\ldots\ar@{-}[ul]&f\overline{\QQ_2}\ar@{-}[u]\ar@{-}[ul]\ar@{-}[ull]\ar@{-}[ur]\ar@{-}[urr]&\ldots\ar@{-}[ur]&\ldots\ar@{-}[u]&\ldots\ar@{-}[ul]&\overline{\QQ_4}\ar@{-}[ull]\\
}
$$
\vskip0.5cm
\subsection{Ramified case}

We assume now that $f$ is ramified in $D$, so $\f=P^2$, for some prime ideal $P$ of $D$. 

\begin{Thm}

\begin{itemize}

\item[(i)] If $d \equiv 1, 2\pmod 4$ or $d \equiv  3 \pmod 4$ and $f \ne 2$, then we have $P = fD+\sqrt{d}D$. If $d \equiv  3 \pmod 4$ and $f = 2$, then $P = 2D+ (1 + \sqrt{d})D$.

\item[(ii)] Let $Q\subseteq \f$ be a basic $\f$-primary ideal. Then either $P^{4}\subset Q\subseteq P^2$ or $P^{5} \subset Q\subseteq P^3$.

\item[(iii)] If $\f \supset Q \supset \f^2$, then either $Q = J_a = (f^2, f (a + \sqrt{d}))$, for some $a = 0, 1, \dots, f-1$, or $Q = J = (f, f^2 \sqrt{d}) = f O$. 

\item[(iv)] if $P^3 \supset Q \supset P^5 = f P^3$, then $Q = H_a = (f^3, a f^2 + f \sqrt{d})$, for some $a = 0, 1, \dots, f-1$, or $Q = (f^2, f^2 \sqrt{d}) = f \f = P^4$, except when $f = 2$ and $d \equiv 3\pmod 4$; in this latter case, we either get $Q = (8, 2(1 + \sqrt{d}))$ or $Q = (8, 4 + 2(1 + \sqrt{d}))$, or $Q = (4, 4(1 + \sqrt{d})) = P^4$.   

\end{itemize}

\end{Thm}

\begin{proof}
(i) In any case, we have $\f = (f, f \sqrt{d})$. Assume that $f \mid d$; we get $d = f \lambda$, with $\lambda \notin f \Z$, since $d$ is square-free. Then the ideal $(f, \sqrt{d})$ of $D$  satisfies $(f, \sqrt{d})^2 = (f^2, d)D = f D = \f$, hence it coincides with $P$. This argument covers all the possible cases, except when $f = 2$ and $d \equiv 3$ modulo $4$.
Under this latter circumstance, we take the ideal $(2, 1 + \sqrt{d})$, whose square is $(4, {1 + d} + 2\sqrt{d}) = (4, 2 \sqrt{d}) = 2 D = \f$, where the preceding equalities hold since $d + 1 \in 4 \Z$, and $d \in (2, \sqrt{d})$ is odd. It follows that $P = (2, 1 + \sqrt{d})$ as required.

(ii) Since $D$ is a Dedekind domain and $Q$ is a basic $\f$-primary ideal, $QD$ is equal either to $P^2$ or to $P^3$. In both cases, by Lemma \ref{index}, $fQD\subset Q\subseteq QD$, which is the statement.

(iii) and (iv) follow from Theorem \ref{intermediate2}, since, by (i), either $P^3 =   P \f =  Pf D = fP = (f^2, f \sqrt{d})$ or $P^3 = 2 P = (4, 2(1 + \sqrt{d}))$, in the exceptional case. In this latter case, we immediately get the equality $(4, 4(1 + \sqrt{d})) = 2 \f = \f^2 = P^4$. 
\end{proof}

Besides the basic elements $t\in\f$ such that $\f^2\subset tO\subset \f$, which are associated to $f$ by a unit of $D$ (see Lemma \ref{units2}), in the ramified case we may have other basic elements such that $P^5\subset tO\subset P^3$, according to whether $P$ is a principal ideal of $D$ or not, as the next result shows.

\begin{Prop}\label{basicramifiedcase}
There exists a basic element $t\in O$ such that $P^5\subset tO\subset P^3$ if and only if $P$ is a principal ideal of $D$. If this condition holds, say $P=\beta D$, for some $\beta\in D$, then every basic element is associated to $f\beta$ by a unit of $D$.
\end{Prop}

\begin{proof}
Let us assume that $P=\beta D$, for some $\beta\in D$. Under the present circumstances we get $N(\beta) = f$ and $f = u \b^2$, for some unit $u \in D$. Clearly, $\beta\notin O$, otherwise $\beta\in P\cap O=\f=P^2$, which is impossible. Hence, $t= f\beta$ is a basic element, according to Proposition \ref{norm}, since its norm is $f^3$ and $t\notin fO$. Since $tD=\beta^3 D=P^3$, we get $\b^5 D = P^5 \subset tO \subset P^3$.

Conversely, let $t\in O$ be a basic element such that $P^5\subset tO\subset P^3$. Using Lemma \ref{index}, we get $tD=P^3=f P$, so $P=\frac{t}{f}D$ is a principal ideal of $D$. 

The last claim follows arguing as in Lemma \ref{units2}.
\end{proof}

The diagram below represents the lattice of $\f$-primary ideals in the ramified case. By Proposition \ref{D-ideal} and the above description of the basic ideals, all the basic ideals, with the exception of $\f$ and $P^3$, are invertible.

$$
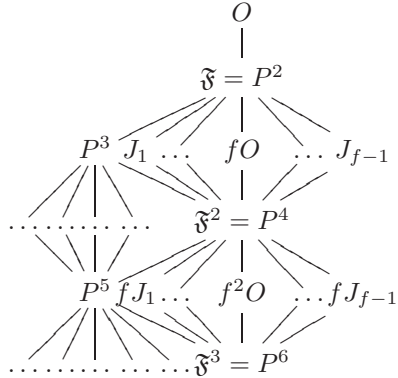
\xymatrix@R-1.1pc@C-2.4pc{
&&&&&&O&&\\
&&&&&&\f=P^2\ar@{-}[u]&&\\
&&&P^3\ar@{-}[urrr]&J_1\ar@{-}[urr]&\ldots\ar@{-}[ur]&fO\ar@{-}[u]&\ldots\ar@{-}[ul]&J_{f-1}\ar@{-}[ull]\\
\ldots\ar@{-}[urrr]&&\ldots\ar@{-}[ur]&\ldots\ar@{-}[u]&\ldots\ar@{-}[ul]&&\f^2=P^4\ar@{-}[u]\ar@{-}[ul]\ar@{-}[ulll]\ar@{-}[ur]\ar@{-}[ull]\ar@{-}[urr]&&\\
&&&P^5\ar@{-}[u]\ar@{-}[ul]\ar@{-}[ur]\ar@{-}[ulll]\ar@{-}[urrr]\ar@{-}[dl]&fJ_1\ar@{-}[urr]\ar@{-}[drr]&\ldots\ar@{-}[ur]\ar@{-}[dr]&f^2O\ar@{-}[u]&\ldots\ar@{-}[ul]&fJ_{f-1}\ar@{-}[ull]\\
\ldots\ar@{-}[urrr]&&\ldots&\ldots\ar@{-}[u]&\ldots\ar@{-}[ul]&\ldots\ar@{-}[ull]&\f^3=P^6\ar@{-}[u]\ar@{-}[ulll]\ar@{-}[ur]\ar@{-}[urr]&&
}
$$
\vskip0.5cm
 In our final remark we make some considerations for the case where $f$ is not prime.

\begin{Rem} \label{finalremark}
We retain the preceding notation, but here we assume that $f$ is not a prime number, say $f = \prod_{i=1}^n f_i^{s_i}$, where the $f_i\in\Z$ are pairwise distinct prime numbers and $s_i\geq0$. Under the present circumstances, it is straightforward to verify that the conductor $\f = (f, f\o)$ is the product $\f = \prod_{i= 1}^n \mathfrak G_i$, where $\mathfrak G_i= (f_i^{s_i}, f\omega )$, for $i=1,\ldots,n$. The $\mathfrak G_i$ are primary ideals of $O$, namely, 
$\sqrt{\mathfrak G_i}= \f_i = (f_i, f\omega)$, where the $\f_i$'s are the prime ideals of $O$ that contain $\f$. Then the lattice of the primary ideals of $O = \Z[f \o]$ is given by the disjoint union of the lattices of the $\f_i$-primary ideals, together with the chains of the powers of the prime ideals $N$ of $O$ that are coprime with $\f$. So we may confine ourselves to a prime ideal $\f_i$, for a fixed $i\in\{1,\ldots,n\}$. It can be easily verified that $\f_i^2 = f_i \f_i$, so also the lattice of the $\f_i$-primary ideals has a structure by layers. The main definitions and several results, proved above for the case of $\f$ prime, can be adapted to $\f_i$-primary ideals. We intend to examine thoroughly  this general case in a coming paper. The main difference with the case of $\f$ prime is that the $\f_i$ are not $D$-modules, and, in fact, no $\f_i$-primary ideal is a $D$-module if $f$ is not a power of a single prime. 

As an instance, we give a generalization of the formula we got in the case of prime conductor to the general case. We use the notation $O_f= \Z[f\omega]$. Then for each $i=1,\ldots, n$ we have
$$\f _i=f_iO_f+f\omega O_f=f_i O_{f/f_i}=(O_f: O_{f/f_i})$$
that is, $\f_i$ is the conductor of the order $O_{f/f_i} = \Z[\frac{f}{f_i}\omega]$ into the order $\Z[f\omega]$.     
\end{Rem}

\end{document}